\documentclass[12pt,reqno]{amsart}
\usepackage{amssymb,a4wide,xy}
\usepackage{amsmath,amsthm,amscd}
\usepackage{tikz-cd}
\usepackage{enumitem}
\xyoption{all}

\newtheorem{theorem}{Theorem}[section]
\newtheorem{proposition}[theorem]{Proposition}
\newtheorem{definition}[theorem]{Definition}

\newtheorem{example}[theorem]{Example}

\newtheorem{remark}[theorem]{Remark}
\newtheorem{lemma}[theorem]{Lemma}

\def\al{\alpha}
\def\be{\beta}

\def\vp{\varphi}

\def\pa{\partial}

\def\gm{\gamma}

\def\lra{\longrightarrow}

\def\ol{\overline}

\def\k{\operatorname{\mathbf{k}}}
\def\Ann{\operatorname{Ann}}

\def\g{\operatorname{\mathfrak{g}}}
\def\q{\operatorname{\mathfrak{q}}}

\DeclareMathOperator{\LB}{\mathtt{Lb}}
\DeclareMathOperator{\XLB}{\mathtt{XLb}}
\DeclareMathOperator{\AS}{\mathtt{As}}
\DeclareMathOperator{\XAS}{\mathtt{XAs}}
\DeclareMathOperator{\U}{\mathtt{U}}
\DeclareMathOperator{\XU}{\mathtt{XU}}
\DeclareMathOperator{\Ud}{\mathtt{U_d}}
\DeclareMathOperator{\XUd}{\mathtt{XU_d}}
\DeclareMathOperator{\I}{\mathtt{I}}
\DeclareMathOperator{\J}{\mathtt{J}}
\DeclareMathOperator{\Up}{\Upsilon}
\DeclareMathOperator{\Ga}{\Gamma}

\DeclareMathOperator{\Liea}{\mathtt{Lie_1}}
\DeclareMathOperator{\Liel}{\mathtt{Lie_2}}
\DeclareMathOperator{\XLiea}{\mathtt{XLie_1}}
\DeclareMathOperator{\XLiel}{\mathtt{XLie_2}}

\def\Im{\operatorname{Im}}
\def\Ker{\operatorname{Ker}}

\def\Hom{\operatorname{Hom}}

\def\As{\operatorname{\textbf{\textsf{As}}}}
\def\XAs{\operatorname{\textbf{\textsf{XAs}}}}
\def\XLie{\operatorname{\textbf{\textsf{XLie}}}}
\def\Lie{\operatorname{\textbf{\textsf{Lie}}}}
\def\Lb{\operatorname{\textbf{\textsf{Lb}}}}
\def\XLb{\operatorname{\textbf{\textsf{XLb}}}}

\def\Di{\operatorname{\textbf{\textsf{Dias}}}}
\def\XDi{\operatorname{\textbf{\textsf{XDias}}}}
\def\CDi{\operatorname{\textbf{\textsf{CDias}}}}
\def\IDi{\operatorname{\textbf{\textsf{IDias}}}}

\def\cat{\operatorname{cat}}

\def\D{\operatorname{D}}
\def\id{\operatorname{id}}

\begin{document}

\title[More on crossed modules of Lie, Leibniz, associative and diassociative]{More on crossed modules of Lie, Leibniz, associative and diassociative algebras}

\author{Jos\'e Manuel Casas}
\address{Department of Applied Mathematics I, University of Vigo, E. E. Forestal, 36005	Pontevedra, Spain}
\email{jmcasas@uvigo.es}
\author{Rafael ~F.~Casado}
\address{Department of Algebra, University of Santiago de Compostela\\
15782 Santiago de Compostela, Spain}
\email{rapha.fdez@gmail.com}
\author{Emzar Khmaladze}
\address{A. Razmadze Mathematical Institute, Tbilisi State University,
Tamarashvili St. 6, 0177 Tbilisi, Georgia}
\email{e.khmal@gmail.com}
\author{Manuel Ladra}
\address{Department of Algebra, University of Santiago de Compostela\\
15782 Santiago de Compostela, Spain}
\email{manuel.ladra@usc.es}

\numberwithin{equation}{section}

\thanks{The authors were supported by Ministerio de Econom\'ia y Competitividad (Spain), grant MTM2013-43687-P (European FEDER support included).
 The third and fourth authors were supported by Xunta de Galicia, grants EM2013/016 and GRC2013-045 (European FEDER support included).
 The third author was also supported by Shota Rustaveli National Science Foundation, grant FR/189/5-113/14.}

\begin{abstract}
 Adjoint functors between the categories of crossed modules of dialgebras and Leibniz algebras are constructed.
  The well-known relations between the categories of Lie, Leibniz, associative algebras and dialgebras are extended to the respective categories of crossed modules.
\end{abstract}
\maketitle

\section{Introduction}

Groups, associative algebras and Lie algebras are related by well-known adjoint functors: the group algebra functor is left adjoint to the unit group functor,
 as well as the universal enveloping algebra functor is left adjoint to the Liezation functor. These classical facts have been extended to the categories of crossed modules of groups,
  associative algebras and Lie algebras in \cite{CaInKhLa} and \cite{CaCaKhLa}.

In the non-commutative framework, when Lie algebras are replaced by Leibniz algebras, the analogous objects to associative algebras are diassociative algebras (or dialgebras, for short),
 introduced and studied by Loday \cite{Lo}. There is an adjunction between the categories of Leibniz algebras and dialgebras,
  which is analogous and related to the one between the categories of Lie and associative algebras.

The aim of this paper is to extend this adjunction to the categories of crossed modules of Leibniz algebras and dialgebras,
 and to establish a connection between that adjunction and the one for the categories of crossed modules of Lie and associative algebras \cite{CaCaKhLa}.

In this paper, following the general notion of crossed module in a category of groups with operations (see \cite{Pa,Po}),
 we introduce crossed modules of dialgebras, study their properties (Section~\ref{S:dia}) and show their equivalence with $\cat^1$-dialgebras and internal categories (Section~\ref{S:equiv}).
 We construct the universal enveloping crossed module of a Leibniz crossed module, which is a crossed module of dialgebras.
  This construction defines a functor, and we prove that it is left adjoint to the naturally defined functor from crossed modules of dialgebras to Leibniz crossed modules.
   We show that this adjunction is a natural generalization of that between dialgebras and Leibniz algebras (Section~\ref{S:univ}).
    Finally, we establish relations between categories of  crossed modules of Lie, Leibniz, associative algebras and dialgebras in terms of commutative diagrams of categories and functors (Section~\ref{S:rel}).

\subsection*{Notations and conventions} Throughout the paper, we fix a commutative ring $\k$ with unit. All algebras are considered over $\k$.
  The categories of Lie, Leibniz and (non-unital) associative algebras will be denoted by $\Lie$, $\Lb$ and $\As$, respectively. We say that a diagram of categories and functors
\[
\xymatrix{
	C_1 \ar[d]_{G_1} \ar[r]^{F_1} & C_2 \ar[d]^{F_2} \\
	C_3 \ar[r]_{G_2} &    C_4
}
\]
is commutative if $F_2\circ F_1=G_2\circ G_1$ or even if there is a natural isomorphism of functors $F_2\circ F_1\cong G_2\circ G_1$.

 \section{Preliminaries}

 \begin{definition}
[\cite{Lo}] An associative dialgebra (or simply dialgebra), also called diassociative algebra by some authors, is a $\k$-module $D$ equipped with two $\k$-linear maps
 \begin{equation*}
 \dashv, \; \vdash \colon D\otimes D\rightarrow D,
  \end{equation*}
  satisfying, for all $x,y,z\in D$, the following axioms
 \begin{align}
 \label{eq1} &(x\dashv y)\dashv z = x\dashv (y\vdash z),\\
\label{eq2} &(x\dashv y)\dashv z = x\dashv (y\dashv z),\\
 \label{eq3} &(x\vdash y)\dashv z = x\vdash (y\dashv z),\\
\label{eq4} &(x\dashv y)\vdash z = x\vdash (y\vdash z), \\
 \label{eq5} &(x\vdash y)\vdash z = x\vdash (y\vdash z).
 \end{align}
 \end{definition}
The maps $\dashv$ and $\vdash$ are called \emph{left and right products}, respectively.
 A \emph{morphism of dialgebras} is a $\k$-linear map preserving both left and right products. We denote by ${\Di}$ the category of dialgebras.

A dialgebra $D$ is called \emph{abelian} if both left and right products are trivial, that is, $x\dashv y=x\vdash y=0$, for all $x,y\in D$. Note that any $\k$-module can be regarded as an abelian dialgebra.

A submodule $I$ of a dialgebra $D$ is called an \emph{ideal} of $D$ if $x\dashv y, x\vdash y, y\dashv x, y\vdash x \in I$ for any $x\in I$ and $y\in D$.

The \emph{annihilator} of a dialgebra $D$ is given by:
\[
\Ann(D)=\{ x\in D \ | \ x\dashv y=y\dashv x=x\vdash y=y\vdash x=0, \ \text{for all} \ y\in D\}.
\]
It is immediate to check that $\Ann(D)$ is indeed an ideal of $D$.

Any associative algebra becomes a dialgebra with $x \dashv y = x y = x \vdash y$, so we have an inclusion functor $\As\hookrightarrow \Di $.
 Let  $\AS \colon \Di\to \As $  denote its left adjoint functor, which assigns to a dialgebra $D$ the quotient of $D$ by the ideal generated
 by all elements $x\dashv y-x\vdash y$, $x,y\in D$ (see \cite[Subsection 2.6]{Lo}).

Let $V$ be a $\k$-module and $T(V)$ denote the tensor module ${T}(V)=\bigoplus_{n\geq 0} V^{\otimes n}$. The \emph{free dialgebra} over $V$
 is explicitly described in \cite[Theorem 2.5]{Lo} as follows: its underlying $\k$-module  is
 $T(V)\otimes V\otimes T(V)$ and the two products are determined by
  \begin{align*}
 (v_{-n} \cdots v_{-1}\otimes v_0\otimes v_1 \cdots v_m ) &\dashv (w_{-p} \cdots w_{-1}\otimes w_0\otimes w_1 \cdots w_q ) \\
 & = v_{-n} \cdots v_{-1}\otimes v_0\otimes v_1 \cdots v_m w_{-p}\cdots w_q,\\
   (v_{-n} \cdots v_{-1}\otimes v_0\otimes v_1 \cdots v_m ) &\vdash (w_{-p} \cdots w_{-1}\otimes w_0\otimes w_1 \cdots w_q ) \\
 & = v_{-n} \cdots v_{m}w_{-p}\cdots w_{-1}\otimes w_0\otimes w_1 \cdots w_q,
   \end{align*}

Let us recall from \cite{Lo3} that a \emph{Leibniz algebra} is a $\k$-module $\g$ equipped with a $\k$-linear map (\emph{Leibniz bracket}) $[ \ ,\ ] \colon \g\otimes \g\to\g$ satisfying the Leibniz identity
\begin{equation}\label{Leibniz}
[x,[y,z]]=[[x,y],z]-[[x,z],y], \quad x,y,z\in \g.
\end{equation}
If the bracket is antisymmetric, then $\g$ is a Lie algebra.

The key point in the introduction of dialgebras is that the bracket
\begin{equation}\label{eqDiasToLb}
[x,y]=x\dashv y - y\vdash x
\end{equation}
 endows the dialgebra with a Leibniz algebra structure. Of course, if left and right products coincide, we get a Lie algebra.
  Thus, we have a functor $\LB \colon \Di\to \Lb$, which admits as left adjoint the functor $\Ud \colon \Lb\to \Di$, that assigns to a Leibniz algebra
   $\mathfrak{g}$ its \emph{universal enveloping dialgebra} (see \cite{Lo}), defined as the following quotient of the free dialgebra over the underlying $\k$-module of $\g$
\[
\Ud(\mathfrak{g})=T(\mathfrak{g})\otimes \mathfrak{g}\otimes T(\mathfrak{g})/ \{[x,y]-x\dashv y+y\vdash x \;\mid \;x,y\in \mathfrak{g}\}.
\]
Therefore, we have the following diagram of categories and functors

\begin{equation}\label{diagram1}
\xymatrix @=20mm {
	\As \ar@<1.1ex>[d]^{\cap} \ar@<-1.1ex>[r]^{\perp}_{\Liea} & \Lie \ar@<-1.1ex>[l]_{\U} \ar@<-1.1ex>[d]^{\vdash}_{\cap} \\
	\Di \ar@<1.1ex>[u]_{\dashv}^{\AS} \ar@<1.1ex>[r]_{\ \top}^{\LB} &   \Lb \ar@<1.1ex>[l]^{\Ud} \ar@<-1.1ex>[u]_{\Liel}
}
\end{equation}
in which the respective outer and inner squares of left and right adjoint functors are commutative.
$\Liea \colon \As\to \Lie$ (resp. $\Liel \colon \Lb\to \Lie$) is the Liezation functor, which assigns to an associative algebra $A$
 its Lie algebra with the Lie bracket $[a,b]=ab-ba$, $a,b\in A$ (resp. to a Leibniz algebra $\g$ the quotient of $\g$ by the ideal generated
  by all elements $[g,g]$, $g\in \g$) and $\U$ is the universal enveloping algebra functor for Lie algebras.

\section{Crossed modules of dialgebras}\label{S:dia}

It is known (see e.g. \cite{Mo}) that the category of dialgebras is a category of interest in the sense of Orzech \cite{Or}, i.e.
 it is a category of groups with operations satisfying two additional axioms (see e.g. \cite{CaDaLa}).
  In \cite{Po}, Porter introduced the concept of crossed module in categories of groups with operations (see also in \cite{CaDaLa}
   the same notion in categories of interest) and showed that crossed modules are equivalent to internal categories.

In this section we examine the notions of action, semi-direct product and crossed module of dialgebras. Of course, our definitions agree with those of \cite{Po}.

\begin{definition}
Let $D$ and $L$ be dialgebras. We will say that $D$ acts on $L$ if four bilinear maps, two of them denoted by the symbol $\dashv$ and the other two by $\vdash$,
\begin{align*}
& \dashv \colon D\otimes L  \rightarrow L, \qquad \dashv \colon L\otimes D  \rightarrow L,\\
&  \vdash \colon D\otimes L  \rightarrow L, \qquad \vdash \colon L\otimes D  \rightarrow L
\end{align*}
are given, such that $30$ equalities hold, which are obtained from the equations \eqref{eq1}--\eqref{eq5}
 by taking one variable in $D$ and two variables in $L$ ($15$ equalities), and one variable in $L$ and two variables in $D$ ($15$ more equalities).

The action is called {trivial} if all these four maps are trivial.
\end{definition}

\begin{example}\label{Example_action}\hfill

\noindent $(i)$  If $ 0\to L \xrightarrow{\sigma}
{E} \xrightarrow{\pi} {D}\to 0 $ is a split short
exact sequence of dialgebras, that is, there exists a
homomorphism of dialgebras $\vp \colon {D}\to {E}$ such that $\pi \vp
=\id_{{D}}$, then there is an action of the dialgebra ${D}$ on ${L}$, defined in the standard
way, by taking left and right products in the dialgebra ${E}$:
\begin{align*}
x\dashv l= \vp(x)\dashv \sigma(l), \qquad l\dashv x=  \sigma(l) \dashv \vp(x), \\
x\vdash l= \vp(x)\vdash \sigma(l), \qquad l\vdash x=  \sigma(l) \vdash \vp(x),
\end{align*}
\noindent for any $x\in D$, $l\in L$.
\;

\noindent $(ii)$ If ${D}$ is a subdialgebra of a
dialgebra ${E}$ (maybe ${D}={E}$) and ${L}$ is an ideal in ${E}$, then the left and right products in ${E}$ yield
an action of ${D}$ on ${L}$.

\;

\noindent $(iii)$ Any homomorphism of dialgebras
${D}\to {L}$ induces an action of ${D}$ on ${L}$, in the standard way, by taking images of elements of
${D}$ and left and right products in ${L}$.

\;

\noindent $(iv)$ If $\mu \colon {L}\to D$ is a surjective morphism of dialgebras with $\Ker\mu$ in the annihilator of
${L}$, then there is an action of $D$ on ${L}$, defined in the standard way, by choosing pre-images of
elements of $D$ and taking left and right products in $L$.

\;

\noindent $(v)$ If ${L}$ is a bimodule over a dialgebra $D$ (for the definition see \cite[Subsection 2.3]{Lo}), thought as an abelian dialgebra,
 then the bimodule structure defines an action of $D$ on the (abelian) dialgebra $L$.

\end{example}

Note that if a dialgebra $D$ acts on a dialgebra $L$, then $L$, as a $\k$-module, has a structure of bimodule over the dialgebra $D$.

\begin{definition}
 Given an action  of a dialgebra $D$ on a dialgebra $L$, we define the semi-direct product $L \rtimes D$ as the dialgebra with underlying $\k$-module $L \oplus D$, endowed with left and right  products given by
\begin{align*}
(l_1,x_1)\dashv (l_2,x_2)= (l_1 \dashv l_2+x_1\dashv l_2 + l_1\dashv x_2, x_1\dashv x_2),\\
(l_1,x_1)\vdash (l_2,x_2)= (l_1 \vdash l_2+x_1\vdash l_2 + l_1\vdash x_2, x_1\vdash x_2),
\end{align*}
for all $x_1,x_2 \in D$ and $l_1, l_2 \in L$.
\end{definition}
Straightforward calculations show that $L \rtimes D$ is indeed a dialgebra and there is a split short exact sequence of dialgebras
\[
0\to L\xrightarrow{i} L \rtimes D \xrightarrow{\pi} D\to 0,
\]
where $i(l) = (l,0)$, $\pi(l,x)= x$ and the splitting $s \colon D\to L \rtimes D$ is given by $s(x)=(0,x)$.
Note that the action of $D$ on $L$ defined by this split short exact sequence, as in Example~\ref{Example_action}(i), coincides with the initial one.

\begin{definition}\label{Definition_crossed}
A crossed module of dialgebras is a homomorphism of dialgebras
$\mu \colon {L}\to D$, together with an action of $D$ on ${L}$ satisfying the following
conditions:
\begin{align}
& \qquad\qquad \mu(x\dashv l)=x\dashv \mu(l), \qquad \mu(x\vdash l)=x\vdash \mu(l),\label{cr1}\tag{cr1}\\
& \qquad\qquad \mu(l\dashv x)=\mu(l)\dashv x, \qquad \mu(l\vdash x)=\mu(l)\vdash x,\label{cr2}\tag{cr2}\\
& \qquad\qquad \mu(l)\dashv l'=l\dashv l'=l\dashv \mu(l'),\label{cr3}\tag{cr3}\\
& \qquad\qquad \mu(l)\vdash l'=l\vdash l'=l\vdash \mu(l'),\label{cr4}\tag{cr4}
\end{align}
for all $x\in D$ and $l,l'\in L$.
\end{definition}

 A \emph{morphism of crossed modules of dialgebras}
$({L} \xrightarrow{\mu} D)\to ({L'} \xrightarrow{\mu'} D')$ is a pair
$(\al,\be)$, where $\al \colon {L}\to {L'}$ and
$\be \colon D\to \D'$ are homomorphisms of dialgebras satisfying:
\begin{enumerate}[label=(\roman*)]
\item $\mu'\al=\be\mu$,
\item $\al$ preserves the action
of $D$ via $\be$, i.e.,
\begin{align*}
&\al(x\dashv l)=\be(x)\dashv \alpha(l),\quad \alpha(l\dashv x)=\al(l)\dashv \be(x),\\
&\al(x\vdash l)=\be(x)\vdash \alpha(l),\quad \alpha(l\vdash x)=\al(l)\vdash \be(x).
\end{align*}
for all $x\in D$ and $l\in L$.
\end{enumerate}

Crossed modules of dialgebras constitute a category which will be denoted by $\XDi$.

\

The first two examples immediately below show that the concept of
crossed module of dialgebras generalizes both the concepts of ideal
and bimodule of a dialgebra.

\begin{example}\hfill

\noindent $(i)$ An inclusion ${L}\hookrightarrow D$ of an ideal
${L}$ of a dialgebra $D$ is a crossed module, where the action of $D$ on ${L}$ is
given by left and right products  in $D$, as in Example~\ref{Example_action}(ii). Conversely, if
$\mu \colon {L}\to D$ is a crossed module of dialgebras which is injective, then, by Lemma~\ref{Lemma_crossed}(ii) below,
${L}$ is isomorphic to an ideal of $D$.

\;

\noindent $(ii)$ For any bimodule $L$ over a dialgebra $D$ the trivial map $0 \colon L\to D$ is a crossed module with the action of $D$
on the (abelian) dialgebra $L$ described in Example~\ref{Example_action}(v). Conversely, if $0 \colon L\to D$ is a crossed module of dialgebras, then
$L$ is necessarily an abelian dialgebra and the action of $D$ on $L$ determines on $L$ a bimodule structure over $D$.

\;

\noindent $(iii)$ Any homomorphism of dialgebras $\mu \colon {L}\to D$, with ${L}$ abelian and
$\Im\mu$ in the annihilator of $D$, provides a crossed module with the trivial action of $D$  on $L$.

\;

\noindent $(iv)$ Any surjective morphism of dialgebras $\mu \colon {L}\to D$, with $\Ker\mu$ in the annihilator of $L$ and the action of $D$ on $L$
described in Example~\ref{Example_action}(iv), is a crossed module of dialgebras.
\end{example}

The following result is an immediate consequence of Definition~\ref{Definition_crossed}.

\begin{lemma}\label{Lemma_crossed}
Let $\mu \colon {L}\to D$ be a crossed module of dialgebras. Then, the following conditions hold:
\begin{enumerate}[label=(\roman*)]
\item The kernel of $\mu$ is contained in the annihilator of $L$.
\item The image of $\mu$ is an ideal in $D$.
\item $\Im\mu$ acts trivially on the annihilator $\Ann(L)$, and so trivially on
$\Ker\mu$. Hence, $\Ker\mu$ inherits an action of $D/\Im\mu$, making $\Ker\mu$ a bimodule over the
dialgebra $D/\Im\mu$.
\end{enumerate}
\end{lemma}

\begin{proposition} \hfill
\begin{enumerate}[label=(\roman*)]
\item A homomorphism of dialgebras $\mu \colon {L}\to D$ is a crossed module if and only if the maps
\begin{equation}\label{eq_ho1}
(\mu,\id_{D})\colon{L}\rtimes D\to D\rtimes D
\end{equation}
and
\begin{equation}\label{eq_ho2}
(\id_{{L}},\mu) \colon {L}\rtimes {L}\to {L}\rtimes D
\end{equation}
are homomorphisms of dialgebras.

\item  If $\mu \colon {L}\to D$ is a crossed module of dialgebras, then
\begin{equation}\label{eq_ho3}
{L}\rtimes D\to {L}\rtimes D, \quad (l,x)\mapsto (-l,\mu( l)+x)
\end{equation}
is a homomorphism of dialgebras.
\end{enumerate}
\end{proposition}

\begin{proof}
\eqref{eq_ho1} (resp. \eqref{eq_ho2} ) is a homomorphism  if and only if the conditions \eqref{cr1} and \eqref{cr2} (resp. \eqref{cr3} and \eqref{cr4} ) hold.
 On the other hand,  \eqref{eq_ho3} is a homomorphism because of the conditions \eqref{cr1}--\eqref{cr4}.
\end{proof}

\section{Equivalences with cat$^1$-dialgebras and internal categories} \label{S:equiv}

By analogy to Loday's original notion of cat$^1$-groups \cite{Lo2}, we state the following definition.
\begin{definition}
A cat$^1$-dialgebra $({E},{D},s,t)$
consists of a dialgebra ${E}$ together with a
subdialgebra ${D}$ and two homomorphisms of
dialgebras $s,t \colon {E}\to {D}$ satisfying the following conditions:
\begin{align*}
& \qquad\qquad s|_{D}=t|_{D}=\id_{D}\;,\label{ct1}\tag{ct1}\\
& \qquad\qquad \Ker s \dashv \Ker t= \Ker t \dashv \Ker s = \Ker s \vdash \Ker t=\Ker t \vdash \Ker s =0\;.\label{ct2}\tag{ct2}
\end{align*}
\end{definition}

\noindent A \emph{morphism of cat$^1$-dialgebras} $({E},{D},s,t)\to({E'},{D'},s',t')$
is a homomorphism of dialgebras $f \colon {E}\to {E'}$ such that $f({D})\subseteq {D'}$ and
$s'f=f|_{D}\, s$, $t'f=f|_{D}\,t$. We denote by
${\CDi}$ the category of cat$^1$-dialgebras.

\

We associate to a cat$^1$-dialgebra $({E},{D},s,t)$ a crossed module of dialgebras as follows: ${L}=\Ker s$, $\mu=t|_{L}$ and the action of ${D}$ on ${L}$ is given by the left and right products in
${E}$ (see Example~\ref{Example_action}(ii)). We claim that $\mu \colon {L}\to D$ is a crossed module. Moreover, we
have the following result:

\begin{lemma}\label{Lemma_cat1}
The above assignment defines a functor $\Phi \colon \CDi\to \XDi$.
\end{lemma}
\begin{proof}
It is easy to check that, with the above notation, $\mu \colon {L}\to D$ is a crossed module. Indeed,
conditions \eqref{cr1} and \eqref{cr2} are trivially verified thanks to \eqref{ct1}, whilst conditions \eqref{cr3} and \eqref{cr4} follow from \eqref{ct2}.
Now we define $\Phi \colon \CDi\to \XDi$ by
\[
\Phi(E,D,s,t)=(\mu \colon {L}\to D) \quad \text{and} \quad \Phi(f)=(f|_{L},f|_{D})\;,
\]
where
$f$ is a morphism in $\CDi$. It is easy to see that
the pair $(f|_{L},f|_{D})$ is indeed a morphism of crossed modules  and $\Phi$ is a functor.
\end{proof}

\

We denote by $\IDi$ the category of
internal categories in ${\Di}$, that is, objects
$({E},{D},s,t,\sigma,\gm)$ in ${\IDi}$ are diagrams of dialgebras of the form
\[
\xymatrix@C=0.5cm{
  {E}\times_{{D}}{E}\ar[r]^{\ \ \ \gm}& {E} \ar@<0.6mm>[rr]^{s}\ar@<-0.6mm>[rr]_{t} &&
   D\ar@/_1.3pc/[ll]_{\sigma}
    }
\]
such that $s\sigma=t\sigma=id_{D}$ and the `operation'
$\gm$ satisfies the usual axioms of a category. Morphisms
$({E},{D},s,t,\sigma,\gm)\to (E,D',s',t',\sigma',\gm')$ are pairs
$(\vp \colon E\to {E'}, \psi \colon D\to{D'})$ of homomorphisms of dialgebras such that  $\psi s=s' \vp$,
$\psi t=t'\vp$, $\vp\sigma=\sigma' \psi$ and $\vp\gm = \gm'
(\vp\times \vp)$.

Given an  internal category $(E,D,s,t,\sigma,\gm)$ in $\Di$, we construct its corresponding crossed module of dialgebras
$\mu \colon {L}\to D$, where ${L}=\Ker s$, $\mu=t|_{L}$ and the action of $D$ on ${L}$ is induced by the homomorphism $\sigma$ (as in Example~\ref{Example_action}(iii)).

\begin{lemma}\label{internal}
The above assignment defines a functor $\Psi \colon \IDi\to \XDi$.
\end{lemma}
\begin{proof}
With the above notation, $\mu \colon {L}\to {D}$ is a
crossed module of dialgebras. Indeed, conditions
\eqref{cr1} and \eqref{cr2} are trivially verified, since $t$ is a homomorphism of
dialgebras and $t \sigma (x)=x$ for all $x\in D$. Since $s\sigma = t
\sigma=\id_{D}$ and by the unit law equivalent for the given
internal category, we get that
\begin{align*}
\gm\big(l, \sigma\mu (l)+l'\big)&=\gm\big(l + \sigma s (l'),
\sigma t(l)+l'\big)=\gm\big(l , \sigma t(l)\big)+\gm\big(\sigma s (l'), l'\big)=l+l'
\end{align*}
for all $l,l'\in L=\Ker s$. Therefore, since $\gm$ is a homomorphism of dialgebras,
we have that
\[
\gm\big((l_1, \sigma\mu (l_1)+l'_1)\star (l_2, \sigma\mu (l_2)+l'_2)\big)=(l_1+l'_1)\star (l_2+l'_2),
\]
where $\star$ denotes either $\dashv$ or $\vdash$.
Then, we can easily deduce that all equalities of (cr3) and
(cr4) are satisfied. For instance, if we take
$l_1=0$, $l'_2=0$, we obtain that $l'_1\star \mu (l_2)=l'_1\star l_2$.

Hence, we define $\Psi \colon {\IDi}\to \XDi$ by
\[
\Psi(E,D,s,t,\sigma,\gm)=(\mu \colon L\to D) \quad {\text{and}} \quad \Psi(\vp,\psi)=(\vp|_{L}, \psi),
\]
where $(\vp,\psi) \colon (E,D,s,t,\sigma,\gm)\to (E',D',s',t',\sigma',\gm')$ is a morphism in the
category $\IDi$. It is straightforward to check that  $(\vp|_{L}, \psi) \colon (L\xrightarrow{\mu} D)\to(L'\xrightarrow{\mu'} D')$ is a morphism of crossed modules of dialgebras.
\end{proof}

\begin{theorem}
The categories $\XDi$, $\CDi$ and $\IDi$ are equivalent.
\end{theorem}
\begin{proof}
We give constructions of the quasi-inverses for the functors $\Phi$ and $\Psi$, defined in Lemma~\ref{Lemma_cat1} and Lemma~\ref{internal}, as
follows.

Given a crossed module of dialgebras $\mu \colon L\to D$, it can be readily checked that the semi-direct product $L\rtimes D$, together with maps $s,t \colon L\rtimes
D\to D$, given by $s(l,x)=x$ and $t(l,x)=\mu(l)+x$ for all $(l,x) \in L \rtimes D$, is a cat$^1$-dialgebra, and the diagram
\[
\xymatrix@C=0.5cm{
(L\rtimes D)\times_{D} (L\rtimes D)\ar[r]^{\qquad \ \ \ \gm}&
{L}\rtimes D
\ar@<0.6mm>[rr]^{s}\ar@<-0.6mm>[rr]_{t} &&
   D\ar@/_1.3pc/[ll]_{\sigma}
    }
\]
is an internal category in $\Di$ whose `operation' $\gm$ is
given by $\gm\big((l,x),(l',x+\mu (l))\big)=(l+l',x)$ and $\sigma(x)=(0,x)$.
Here observe that any element of $(L\rtimes D)\times_{D} (L\rtimes D)$ is of the form $\big((l,x),(l',x+\mu(l))\big)$.

The assignments $(L\xrightarrow{\mu} D)\to (L\rtimes D,D, s,t)$ and $(L\xrightarrow{\mu} D )\to (L\rtimes D,D, s,t,\sigma,\gm)$ are clearly functorial and
provide the required quasi-inverses for the functors $\Phi$ and $\Psi$, respectively.
\end{proof}

\section{Universal enveloping crossed module of a Leibniz crossed module}\label{S:univ}

In this section we extend the functors $\LB$ and $\Ud$ from diagram~\ref{diagram1} to the categories of crossed modules of dialgebras and Leibniz algebras in such a way that those extensions preserve the adjunction.

\subsection{Leibniz crossed modules} First, we recall some needed notations and facts about crossed modules of Leibniz algebras from \cite{CaKhLa, LoPi}.

Let $\g$ and $\q$ be Leibniz algebras. A \emph{Leibniz action} of $\g$ on $\q$ is a couple of $\k$-linear
maps $\g\otimes \q\to \q , \ \q\otimes \g\to \q$, both denoted again by the symbol $[\;,\;]$, such that $6$ equalities hold. Those equalities are obtained from the Leibniz identity \eqref{Leibniz} by taking one variable in $\g$ and two variables in $\q$ ($3$ equalities), and one variable in $\q$ and two variables in $\g$ ($3$ more equalities).

A Leibniz action of $\g$ on $\q$ enables us to define the \emph{semi-direct product}  $\q\rtimes\g$
as the Leibniz algebra with underlying $\k$-module $\q\oplus\g$ and the Leibniz bracket
\[
[(q,g),(q',g')]=([q,q']+[q,g']+[g,q'], [g,g']), \quad   q,q'\in \q, \;\;  g,g'\in \g.
\]

A \emph{crossed module of Leibniz algebras} (or \emph{Leibniz crossed module}, for short) is a homomorphism of Leibniz algebras $\nu \colon \q\to\g$, together with an action of $\g$ on $\q$ such that
\begin{align*}
&\nu[g,q]=[g,\nu(q)], \qquad \nu[q,g]=[\nu(q),g],\\
&[\nu(q),q']=[q,q']=[q,\nu(q')],
\end{align*}
for all $g\in \g$, $q,q'\in \q$.

A \emph{morphism of Leibniz crossed modules} $(\al, \be) \colon (\q \xrightarrow{\nu}\g)\to (\q'\xrightarrow{\nu'}\g')$
is a pair of Leibniz homomorphisms $\al \colon \q\to \q'$, $\be \colon \g\to \g'$, such that
$\nu'\al=\be\nu$ and $\al[g,q]=[\be(g),\al(q)]$, $\al[q,g]=[\al(q),\be(g)]$ for all $g\in \g$, $q\in \q$.
We denote by $\XLb$ the category of Leibniz crossed modules.

There are several equivalent descriptions of the category $\XLb$. We mention here the equivalences with simplicial Leibniz algebras whose Moore complexes are of length $1$, with internal categories in the category of Leibniz algebras and with $\cat^1$-Leibniz algebras (see e.g. \cite{CaKhLa0} for equivalences in the more general context of Leibniz $n$-algebras). The nature of these equivalences is reminiscent of the interplays between similar objects in groups \cite{Lo2}.
The equivalence between Leibniz crossed modules and $\cat^1$-Leibniz algebras will be used in the sequel and we shall give some of its details immediately below.

A \emph{$\cat^1$-Leibniz algebra} $(\g_1,\g_0,s,t)$ consists of a Leibniz algebra $\g_1$,
together with a Leibniz subalgebra $\g_0$ and structural homomorphisms $s,t \colon \g_1\to \g_0$  satisfying
\[
s|_{\g_0} = t|_{\g_0} = \id_{\g_0} \qquad \text{and} \qquad [\Ker s, \Ker t]= 0 =[\Ker t, \Ker s].
\]

Given a Leibniz crossed module $\q\xrightarrow{\nu}\g$,  the corresponding $\cat^1$-Leibniz algebra is  $(\q\rtimes \g, \g, s, t)$, where
$s(q, g) = g$, $t(q, g) = \nu(q)+g$ for all $(q, g) \in \q\rtimes \g$. On the other hand,
for a  $\cat^1$-Leibniz algebra $(\g_1,\g_0,s,t)$, the corresponding Leibniz crossed module is
$t|_{\Ker s} \colon \Ker s \to \g_0 $, with the action of $\g_0$ on $\Ker s$ given by the Leibniz bracket in $\g_1$.

\subsection{From crossed modules of dialgebras to Leibniz crossed modules}
Now we extend the functor $\LB \colon \Di\to \Lb$ to the categories of crossed modules $\XDi$ and $\XLb$. We need the following:
\begin{lemma}\label{lemmaXDiasToXLb}
An action of a dialgebra $D$ on another dialgebra $L$ allows to define a Leibniz action of $\LB(D)$ on $\LB(L)$ by:
\[
[x,l]=x\dashv l-l\vdash x \quad \text{and} \quad [l,x]=l\dashv x-x\vdash l, \qquad x\in D, \ l\in L.
\]
\end{lemma}
\begin{proof} The straightforward calculations in the proof of \cite[Proposition 4.2]{Lo}, showing that the bracket \eqref{eqDiasToLb}
satisfies the Leibniz identity, can be repeated if one variable is taken from $D$ (resp. $L$) and two variables are taken from $L$ (resp. $D$).
\end{proof}

It is easy to see that if we apply the functor $\LB \colon \Di\to \Lb$ to a crossed module of dialgebras $\mu \colon L\to D$, we get a Leibniz crossed module
 $\LB(\mu) \colon \LB(L)\to \LB(D)$, where the Leibniz action of $\LB(D)$ on $\LB(L)$ is given as in Lemma~\ref{lemmaXDiasToXLb}.
Moreover, this assignment defines a functor $\XLB \colon \XDi \to \XLb$, which is a natural generalization of the functor $\LB \colon \Di \to \Lb$ in the following sense.

There are full embeddings
\[
\J_0, \J_1 \colon \Di\lra\XDi \qquad (\text{resp.} \quad \J'_0, \J'_1 \colon \Lb\lra\XLb)
\]
defined, for a dialgebra $D$ (resp. for a Leibniz algebra $\g$), by $\J_0(D)=(0\xrightarrow{ \ 0 \ } D)$, $\J_1(D)=(D\xrightarrow{\id_D} D)$
\big(resp. $\J'_0(\g)=(0\xrightarrow{\ 0 \ } \g)$, $\J'_1(\g)=(\g\xrightarrow{\id_{\g}}{\g} )$\big) and
 it is immediate to see that we have the following commutative diagram
\begin{equation}\label{diagramLbDias1}
\xymatrix {
 \Di \ar[d]_{\LB}\ar[r]^{\J_i} & \XDi \ar[d]^{\XLB}  \\
\Lb \ar[r]_{\J'_i} & \XLb
 }
\end{equation}
for $i=0, 1$.
\begin{remark}\label{remark_adj_Di_Lb}
	Let us define three functors from $\XDi$ to $\Di$ (resp. from $\XLb$ to $\Lb$), given by $\Up_0({L}\xrightarrow{\mu} D)=D\slash \mu (L)$, $\Up_1({L}\xrightarrow{\mu} D)=D$
 and $\Up_2({L}\xrightarrow{\mu} D)=L$ (resp. $\Up'_0(\q\xrightarrow{\ \nu \ }{\g} )=\g\slash \nu (\q)$, $\Up'_1(\q\xrightarrow{\ \nu \ }{\g} )=\g$
 and $\Up'_2(\q\xrightarrow{\ \nu \ }{\g} )=\q$). It is straightforward to check that $\Up_i$ is left adjoint to $\J_i$ and $\J_i$ is left adjoint to
  $\Up_{i+1}$ (resp. $\Up'_i$ is left adjoint to $\J'_i$ and $\J'_i$ is left adjoint to $\Up'_{i+1}$), for $i=0,1$.
\end{remark}

\subsection{Universal enveloping crossed module}

In this subsection we construct a left adjoint functor to the functor $\XLB$, which generalizes the universal enveloping dialgebra functor $\Ud \colon \Lb\to\Di$ to crossed modules.

 Let $\nu \colon \q\to \g$ be a Leibniz crossed module and consider its corresponding $\cat^1$-Leibniz algebra
$\xymatrix{ \q \rtimes \g \ar@<0.4ex>[r]^-{s} \ar@<-0.4ex>[r]_-{t} & \g }$. Then, if we apply the universal enveloping dialgebra functor $\Ud$, we obtain a diagram of dialgebras
\begin{equation}\label{diagramNoCat}
\xymatrix{ \Ud(\q \rtimes \g) \ar@<0.4ex>[r]^-{\Ud(s)}
\ar@<-0.4mm>[r]_-{\Ud(t)} & \Ud(\g)}.
\end{equation}

In spite of the fact that $\Ud(s)|_{\Ud(\g)}=\Ud(t)|_{\Ud(\g)}=\id_{\Ud(\g)}$, \eqref{diagramNoCat} is not a $\cat^1$-dialgebra, since condition \eqref{ct2}
in the definition of a $\cat^1$-dialgebra is not fulfilled in general. For this reason, we consider the following diagram,
\begin{equation}\label{diagramYesCat}
\xymatrix{ {\Ud(\q \rtimes \g)}/{X} \ar@<0.4ex>[r]^-{{\ol\Ud(s)}}
\ar@<-0.4mm>[r]_-{{\ol\Ud(t)}} & \Ud(\g)
 },
\end{equation}
where
\begin{align}\label{eqX}
\nonumber X=\Ker\Ud(s)\!\dashv\! \Ker\Ud(t)+ \Ker\Ud(t)\!\dashv\! \Ker\Ud(s)&+\Ker\Ud(s)\!\vdash\! \Ker\Ud(t)\\
&+ \Ker\Ud(t)\!\vdash\! \Ker\Ud(s),
\end{align}
with $\ol\Ud(s)$ and  $\ol\Ud(t)$ induced by $\Ud(s)$ and $\Ud(t)$, respectively.

 Obviously, \eqref{diagramYesCat} is a  $\cat^1$-dialgebra and we define $\XUd(\q \xrightarrow {\ \nu \ }\g)$ to be the crossed module of dialgebras corresponding to this $\cat^1$-dialgebra, that is,
\[
\XUd(\q\xrightarrow {\ \nu \ }\g)=\big(\ol\Ud(t)|_{\Ker\ol\Ud(s)} \colon \Ker\ol\Ud(s) \lra \Ud(\g) \big).
\]
\begin{definition}
Given a Leibniz crossed module $(\q\xrightarrow {\ \nu \ }\g)$, the crossed module of dialgebras $\XUd(\q\xrightarrow {\ \nu \ }\g)$ is called
the universal enveloping crossed module of $(\q\xrightarrow {\ \nu \ }\g)$.
\end{definition}

It is easy, and we leave it to the reader, to check that the universal enveloping crossed module construction for a Leibniz crossed module provides a functor $\XUd \colon \XLb\to\XDi$, which is a natural generalization of the functor $\Ud\colon \Lb\to\Di$ in the sense of the following proposition.
\begin{proposition} The following diagram of categories and functors
\begin{equation}\label{diagramLbDias2}
\xymatrix {
 \Lb \ar[d]_{\Ud}\ar[r]^{\J'_i} & \XLb \ar[d]^{\XUd}  \\
\Di \ar[r]_{\J_i} & \XDi
 }
\end{equation}
is commutative for $i=0,1$.
\end{proposition}
\begin{proof}
For $i=0$, the equality $\J_0\circ \Ud =  \XUd\circ \J'_0$ is trivially verified. Now we show that there is a natural isomorphism of functors
$
\J_1\circ \Ud \cong  \XUd\circ \J'_1.
$

Given a Leibniz algebra $\g$, we have that $(\J_1\circ \Ud) (\g)=(\id_{\Ud(\g)} \colon \Ud(\g)\to\Ud(\g))$.
Thus, we need to show that the crossed module of dialgebras $\XUd(\id_{\g} \colon \g \to \g ) $ is isomorphic to $(\id_{\Ud(\g)} \colon \Ud(\g)\to\Ud(\g))$.

For the Leibniz crossed module $\id_{\g} \colon \g\to\g$, the corresponding $\cat^1$-Leibniz algebra is $(\g\rtimes \g, \g,s,t)$, with $s(g,g')=g'$, $t(g,g')=g+g'$, and there is also a homomorphism $\epsilon \colon \g\to \g\rtimes \g$, $g\mapsto (g,0)$, satisfying that $s \epsilon=0$ and $t\epsilon=\id_{\g}$.
Next, we need to consider the $\cat^1$-dialgebra $\big(\Ud(\g\rtimes \g)/X,\Ud(\g),\ol\Ud(s),\ol\Ud(t)\big)$, where $X$ is the same as in \eqref{eqX}.

Let $\pi \colon \Ud(\g\rtimes \g)\to \Ud(\g\rtimes \g)/X$ denote the canonical epimorphism. Then, we have the equalities
$\ol\Ud(s)\pi\Ud(\epsilon)=\Ud(s)\Ud(\epsilon)=0$ and $\ol\Ud(t)\pi\Ud(\epsilon)=\Ud(t)\Ud(\epsilon)=\id_{\g}$.
This means that $\pi\Ud(\epsilon)$ takes values in $\Ker \ol\Ud(s)$ and
it is a right inverse for the homomorphism $\ol\Ud(t)|_{\Ker \ol\Ud(s)} \colon \Ker \ol\Ud(s)\to \Ud(\g)$. We claim that it is a left inverse as well.
Indeed, we have that $\Ker \ol\Ud(s)= \Ker\Ud(s)/X$ and, as a $\k$-module, $\Ker \Ud(s)$ is generated by all elements of the form
\begin{equation}\label{eq-element}
(g_{-n},g'_{-n})\otimes \dots \otimes (g_i,0)\otimes\dots\otimes (g_m,g'_m),
\end{equation}
$m,n\geq 0$, $-n\leq i\leq m$.
The value of $\Ud(\epsilon)\Ud(t)$ on \eqref{eq-element} is
\begin{equation}\label{eq-element1}
(g_{-n}+g'_{-n},0)\otimes \dots \otimes (g_i,0)\otimes\dots\otimes (g_m+g'_m,0).
\end{equation}
Then, we easily derive the following equalities in $\Ker\Ud(s)/X$:
\begin{align*}
&(g_{-n}+g'_{-n},0)\otimes \dots \otimes (g_i,0)\otimes\dots\otimes (g_m+g'_m,0)\\
=&(g_{-n},g'_{-n})\otimes \dots \otimes (g_i,0)\otimes\dots\otimes (g_m+g'_m,0)\\
=& \cdots \\
=& (g_{-n},g'_{-n})\otimes \dots \otimes (g_i,0)\otimes\dots\otimes (g_m,g'_m).
\end{align*}
Hence, the elements \eqref{eq-element} and \eqref{eq-element1} are equal in $\Ker\Ud(s)/X$ and it follows that  \[\pi\Ud(\epsilon)\ol\Ud(t)|_{\Ker \ol\Ud(s)}=\id_{\Ker \ol\Ud(s)}.\]

Now it is easy to see that the pair $(\ol\Ud(t)\!\!\mid_{\Ker \ol\Ud(s)}, \id_{\Ud(\g)})$ is an isomorphism  between the crossed modules of dialgebras
$(\ol\Ud(t)|_{\Ker \ol\Ud(s)} \colon \Ker\ol\Ud(s)\to\Ud(\g))$  and $(\id_{\Ud(\g)} \colon \Ud(\g)\to\Ud(\g))$, which provides the required isomorphism between the functors  $\J_1\circ \Ud$ and $\XUd\circ \J'_1$.
\end{proof}

\subsection{Adjunction between $\XLb$ and $\XDi$}

The following result extends to the categories of crossed modules the adjunction between the categories $\Lb$ and $\Di$ obtained in \cite{Lo}.

\begin{theorem}
The functor $\XUd$  is left adjoint to the functor $\XLB$.
\end{theorem}
\begin{proof}
We must construct a natural bijection
\[
\Hom_{\XLb}\big((\g\xrightarrow {\ \nu \ }\q), \XLB(L\xrightarrow {\ \mu \ } D)\big)\cong
\Hom_{\XDi}\big(\XUd(\g\xrightarrow {\ \nu \ }\q), (L\xrightarrow {\ \mu \ } D)\big)
\,,
\]
for any $(\g\xrightarrow {\ \nu \ }\q)\in \XLb$ and $(L\xrightarrow {\ \mu \ } D)\in \XDi$.

Given a morphism $(\al,\be) \colon \big((\g\xrightarrow {\ \nu \ }\q) \to  \XLB(L\xrightarrow {\ \mu \ } D)\big)$ of Leibniz crossed modules, consider the corresponding morphism of $\cat^1$-Leibniz algebras
\[
\xymatrix{ \q \rtimes \g \ar[d]_{\al'} \ar@<0.4mm>[r]^-{s}
\ar@<-0.4ex>[r]_-{t} &\  \g    \ar[d]^{\beta}\;
\\
\LB(L \rtimes D) \ar@<0.4ex>[r]^-{\sigma} \ar@<-0.4ex>[r]_-{\tau}
& \LB(D) \ ,
}
\]
where $\al'(q,g)=(\al(q),\be(g))$, for $(q,g)\in \q\rtimes \g$.
Since the functor $\Ud$ is left adjoint to the functor $\LB$, we get the induced commutative diagram of dialgebras
\[
\xymatrix{ \Ud(\q \rtimes \g) \ar[d]_{{\al'}^*} \ar[r]^-{\Ud(s)}
 & \Ud(\g) \ar[d]^{\beta^*}
\\
L \rtimes D \ar[r]^-{\sigma}  & D \;.
 }
\]
Besides, the similar diagram obtained by replacing $\Ud(s)$ by $\Ud(t)$ and $\sigma$ by $\tau$, is also commutative.
Since  $\Ker \sigma\dashv\Ker \tau=\Ker \tau\dashv\Ker\sigma=\Ker \sigma\vdash\Ker \tau=\Ker \tau\vdash\Ker\sigma=0$, we have a uniquely defined morphism of $\cat^1$-dialgebras
\[
\xymatrix{\Ud(\q\rtimes \g )/X \ar[d]_{\ol{{\al'}^*}} \ar@<0.4ex>[r]^-{{\ol\Ud(s)}} \ar@<-0.4ex>[r]_-{{\ol\Ud(t)}} & \Ud(\g) \ar[d]^{\be^*}\;
\\
L \rtimes D \ar@<0.4ex>[r]^-{\sigma} \ar@<-0.4ex>[r]_-{\tau} & D
 }
\]
(here $X$ is the same as in \eqref{eqX}), which corresponds to a uniquely defined morphism from
$\Hom_{\XDi}\big(\XUd(\g\xrightarrow {\ \nu \ }\q), (L\xrightarrow {\ \mu \ } D)\big)$. The inverse assignment is obvious.
\end{proof}

\section{Relations between crossed modules of Lie, Leibniz, associative algebras and dialgebras}\label{S:rel}

In this section we extend the commutative diagram \eqref{diagram1} to the diagram of respective categories of crossed modules.

\subsection{Crossed modules of associative and Lie algebras}

We begin by recalling some basic notions about crossed modules of associative and Lie algebras from \cite{CaCaKhLa} (see also \cite{DoInKhLa, El1, KaLo, Lu}).

A \emph{crossed module of associative algebras} is an algebra homomorphism $\rho \colon R\to A$ together with an action of $A$ on
$R$ such that the following conditions hold:
\begin{align*}
&\rho(a r) = a \rho(r), \qquad  \quad \rho(r a)=\rho(r) a,\\
&\rho(r) r' = r r' = r \rho(r'),
\end{align*}
for all $a\in A$, $r, r' \in R$. Here, by an \emph{action} of $A$ on $R$ we mean an $A$-bimodule structure on $R$ satisfying the following conditions:
\[
a (r r')=(a r) r',\qquad (r a) r'= r (a r'), \qquad (r r')  a = r (r'  a),
\]
for all $a\in A$, $r\in R$.

A \emph{morphism of crossed modules of associative algebras}  $(\al,\be)\colon (R,A,\rho)\to(R',A',\rho')$  is a pair of algebra homomorphisms $(\al \colon R\to R', \; \be \colon A\to A')$,
 such that $\rho'\al=\be\rho$, $\al(a r)=\be(a)\al(r)$ and $\al(r a)=\al(r) \be(a)$, for all $a\in A$, $r\in R$. We denote the category of crossed modules of associative algebras by $\XAs$.

A \emph{Lie crossed module} is a Lie homomorphism $\pa \colon M\to P$, together with a Lie action of $P$ on $M$, such that, for all $m, m'\in M$ and $p \in P$,
\[
\pa[p,m]=[p,\pa(m)],\qquad
[\pa(m),{m'}]=[m,m'].
\]
Here,  by a  \emph{Lie action} of $P$ on $M$ we mean a $\k$-linear map $P\otimes M\to M$, $(p,m)\mapsto [p,m]$ satisfying
\[
[[p,p'],m]= [p,[p',m]]-[p',[p,m]],\qquad
[p,[m,m']]= [[p,m],m']+[m,[p,m']].
\]
 A \emph{morphism of Lie crossed modules} $(\al, \be) \colon (M\xrightarrow {\ \pa \ }P)\to (M'\xrightarrow {\ \pa' \ }P')$
is a pair of Lie homomorphisms $\al \colon M\to M'$, $\be \colon P\to P'$, such that $\pa'\al=\be\pa$ and $\al[p,m]=[\be(p),\al(m)]$, for all $m\in M$, $p\in P$. We denote by $\XLie$ the category of Lie crossed modules.

There are full embeddings
\[
\I_0, \I_1 \colon \As\lra\XAs  \qquad (\text{resp.}\  \I'_0, \I'_1 \colon \Lie\lra\XLie)
\]
defined, for an associative algebra $A$  (resp. for a Lie algebra $P$), by $\I_0(A)=(0\xrightarrow {\ 0 \ } A)$, $\I_1(A)=(A\xrightarrow{\id_A}A)$
(resp.  $\I'_0(P)=(0\xrightarrow{\ 0 \ } P)$, $\I'_1(P)=(P\xrightarrow{\id_P} P)$).
\begin{remark}\label{remark_adj_As_Lie}
	Analogously to the situation described in Remark~\ref{remark_adj_Di_Lb}, we can define three functors from $\XAs$ to $\As$ (resp. from $\XLie$ to $\Lie$),
 given by $\Ga_0({R}\xrightarrow{\rho} A)=A\slash \rho (R)$, $\Ga_1({R}\xrightarrow{\rho} A)=A$ and $\Ga_2({R} \xrightarrow{\rho} A)=R$
 (resp. $\Ga'_0(M\xrightarrow{\ \pa \ }P)=P\slash \pa (M)$, $\Ga'_1(M\xrightarrow{\ \pa \ }P)=P$ and $\Ga'_2(M\xrightarrow{\ \pa \ }P)=M$).
  It is straightforward to check that $\Ga_i$ is left adjoint to $\I_i$ and $\I_i$ is left adjoint to $\Ga_{i+1}$
   (resp. $\Ga'_i$ is left adjoint to $\I'_i$ and $\I'_i$ is left adjoint to $\Ga'_{i+1}$), for $i=0,1$.
\end{remark}
\

In \cite{CaCaKhLa}, a pair of adjoint functors between the categories $\XAs$ and $\XLie$ is constructed,
which extends the classical adjunction
$
\xymatrix  {
\As \ar@<-1.1ex>[r]^{\perp}_{\Liea} & \Lie \ar@<-1.1ex>[l]_{\U}.
}
$
Namely, the functor $\XLiea \colon \XAs\to\XLie$ and its left adjoint $\XU \colon \XLie \to \XAs$ are constructed in such a way that, in the following diagram
\begin{equation}\label{diagramAsLie}
\xymatrix @=20mm {
	\As \ar@<1.1ex>[d]_{\I_i} \ar@<-1.1ex>[r]^{\perp}_{\Liea} & \Lie \ar@<-1.1ex>[l]_{\U} \ar[d]^{\I'_i} \\
	\XAs  \ar@<1.1ex>[r]_{\top}^{\XLiea} & \XLie \ar@<1.1ex>[l]^{\XU}
}
\end{equation}
both inner and outer diagrams are commutative  for $i=0, 1$.

\subsection{$\XAs$ vs $\XDi$}
Let $A$ and $R$ be associative algebras. Then, any action of $A$ on $R$, as associative algebras, defines an action of $A$ on $R$, as dialgebras,
 by letting $a\dashv r = a\vdash r= ar$ and $r \dashv a = r\vdash a= ra$,  for all $a\in A$, $r\in R$. Moreover, any crossed module of associative algebras is,
  at the same time, a crossed module of dialgebras and we have an inclusion functor $\XAs\hookrightarrow \XDi$.

On the other hand, given a crossed module of dialgebras $\mu \colon L\to D$, we get a crossed module of associative algebras $\overline{\AS}(L)\to \AS(D)$,
 where $\overline{\AS}(L)$ is the quotient of $L$ by the ideal generated by all elements of the form $l \dashv l' −- l \vdash l'$,
  $x \dashv l -− x \vdash l$ and $l \dashv x −- l \vdash x$, $l,l'\in L$, $x\in X$, and  the action of the associative algebra $\AS(D)$ on $\overline{\AS}(L)$ is given by
\[
\overline{x}\;\overline{l}=\overline{x\dashv l} \ (=\overline{x\vdash l}) \qquad \text{and} \qquad \overline{l}\;\overline{x}=\overline{l\dashv x} \ (=\overline{l\vdash x}),
\]
where $\overline{x}$ (resp. $\overline{l}$) is an element of $\AS(D)$ (resp. $\overline{\AS}(L)$) represented by $x\in D$ (resp. $l\in L$).
Moreover, the pair of canonical projections $(\pi_1, \pi_2) \colon (L\xrightarrow{\ \mu \ }D )\twoheadrightarrow (\overline{\AS}(L){\lra}\AS(D)) $
 is a morphism of crossed modules of dialgebras, and it is universal among all morphisms from $(L \xrightarrow {\ \mu \ }D )$ to a crossed module of associative algebras.
Thus, this construction defines a functor $\XAS \colon \XDi\to\XAs$, which is left adjoint to the inclusion functor,
  and it  is a natural extension of the functor $\AS \colon \Di\to\As$ in the sense that the following inner and outer diagrams
\begin{equation}\label{diagramAsDias}
\xymatrix @=20mm {
	\As \ar@<-1.1ex>[r]^{\perp}_{\subset} \ar[d]_{\I_i} & \Di  \ar@<-1.1ex>[l]_{\AS} \ar[d]^{\J_i} \\
	\XAs \ar@<1.1ex>[r]_{\top}^{\subset}  & \XDi  \ar@<1.1ex>[l]^{\XAS}
}
\end{equation}

are commutative for $i=0,1$.

\subsection{$\XLie$ vs $\XLb$ }

Let $P$ and $M$ be Lie algebras. Any Lie action of $P$ on $M$ defines in a natural way a Leibniz action of $P$ on $M$, and any Lie crossed module $\pa \colon M\to P$ is,
 at the same time, a crossed module of Leibniz algebras (see \cite[Remark 11]{CaKhLa}). Thus, we have the inclusion functor
$\XLie\hookrightarrow\XLb$.

Conversely, given a Leibniz crossed module $\nu \colon \g\to\q$, we get a Lie crossed module  $\overline{\Liel}(\g)\to\Liel(\q)$,
 where $\overline{\Liel}(\g)$ is the quotient of $\g$ by the ideal generated by all elements $[g,g]$ and $[q,g]+[g,q]$, $g\in \g$, $q\in q$,
  together with the induced Lie action of $\Liel(\q)$ on $\overline{\Liel}(\g)$. This construction defines a functor $\XLiel \colon \XLb\to\XLie$,
   which is left adjoint to the inclusion functor,  and it  is a natural extension of the functor $\Liel \colon \Lb\to\Lie$ in the sense that the following inner and outer diagrams

\begin{equation}\label{diagramLieLb}
\xymatrix @=20mm {
	\Lie \ar@<-1.1ex>[r]^{\perp}_{\subset} \ar[d]_{\I'_i} & \Lb  \ar@<-1.1ex>[l]_{\Liel} \ar[d]^{\J'_i} \\
	\XLie \ar@<1.1ex>[r]_{\top}^{\subset}  & \XLb  \ar@<1.1ex>[l]^{\XLiel}
}
\end{equation}
are commutative for $i=0,1$.

\subsection{Extended diagram for categories of crossed modules} By using the commutative diagrams \eqref{diagramLbDias1},
\eqref{diagramLbDias2}, \eqref{diagramAsLie}, \eqref{diagramAsDias} and \eqref{diagramLieLb}, the commutative diagram \eqref{diagram1}
 can be extended to crossed modules. Moreover, we have the following result:

\begin{theorem} In the following parallelepiped of categories and functors
\

\centering
\includegraphics[width=9cm]{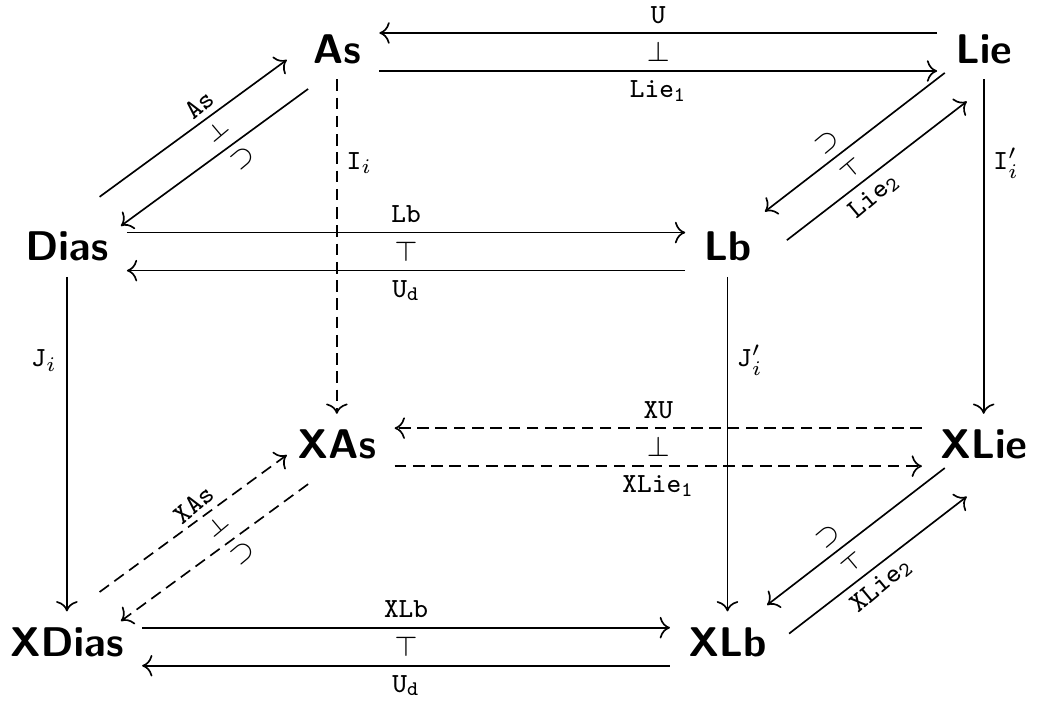}
\end{theorem}
\noindent \emph{all faces are commutative (each one in two different directions), for} $i=0,1$.
\begin{proof}
We only need to check the commutativity of the base of the parallelepiped. The equality  $\subset \circ \XLiea=\XLB\circ \subset$ is trivially verified.
 As composition of left adjoint functors, $\XAS\circ \XUd$ is left adjoint to $\XLB\circ \subset \colon \XAs\to \XLb$ and analogously,
   $\XU\circ \XLiel$ is left adjoint to $\subset \circ \XLiea \colon \XAs\to \XLb$. Since $\subset \circ \XLiea=\XLB\circ \subset$,
   it follows that $\XAS\circ \XUd$ and $\XU\circ \XLiel$ are naturally isomorphic.
\end{proof}
\begin{remark}
	It would be possible to include the adjunctions described in Remarks~\ref{remark_adj_Di_Lb} and \ref{remark_adj_As_Lie}
in the above parallelepipeds. By doing so, we would get four different options for the lateral faces (instead of two), which would trivially commute.
\end{remark}

\end{document}